\documentclass[11pt]{amsart}
\usepackage[utf8]{inputenc}

\title[Trace Paley-Wiener theorem for asymptotic Hecke algebra]{Trace Paley-Wiener theorem for Braverman-Kazhdan's asymptotic Hecke algebra}

\author{Kenta Suzuki}

\calclayout

\usepackage{amsmath, amssymb, amsthm,tikz-cd,mathrsfs,hyperref,colonequals}
\usepackage[shortlabels]{enumitem}

\numberwithin{equation}{section}

\newtheorem{thm}{Theorem}[section]

\newtheorem{lemma}[thm]{Lemma}
\newtheorem{cor}[thm]{Corollary}
\newtheorem{prop}[thm]{Proposition}

\theoremstyle{definition}
\newtheorem{rmk}{Remark}

\newtheorem{defn}{Definition}
\newtheorem{example}[thm]{Example}

\newcommand{\C}{{\mathbb C}}

\newcommand{\fs}{{\mathfrak s}}
\newcommand{\fd}{{\mathfrak d}}
\newcommand{\cO}{{\mathcal O}}
\newcommand{\cP}{{\mathcal P}}

\newcommand{\bG}{{\mathbf G}}
\newcommand{\bM}{{\mathbf M}}
\newcommand{\bP}{{\mathbf P}}
\newcommand{\cH}{{\mathcal H}}
\newcommand{\cC}{{\mathcal C}}
\newcommand{\cJ}{{\mathcal J}}

\newcommand{\GL}{{\mathrm{GL}}}
\newcommand{\sm}{{\mathrm{sm}}}

\DeclareMathOperator{\Cent}{Z}

\DeclareMathOperator{\triv}{triv}
\DeclareMathOperator{\End}{End}

\DeclareMathOperator{\Spec}{Spec}
\DeclareMathOperator{\Vect}{Vect}
\DeclareMathOperator{\Rep}{Rep}
\DeclareMathOperator{\Irr}{Irr}
\DeclareMathOperator{\Forg}{Forg}
\DeclareMathOperator{\Hom}{Hom}
\DeclareMathOperator{\Frac}{Frac}

\DeclareMathOperator{\reg}{reg}

\begin{document}

\begin{abstract}
    Let $\mathbf G$ be a reductive algebraic group over a non-archimedean local field $F$ of characteristic zero and let $G=\mathbf G(F)$ be the group of $F$-rational points. Let $\mathcal H(G)$ be the Hecke algebra and let $\mathcal J(G)$ be the asymptotic Hecke algebra, as defined by Braverman and Kazhdan. We classify irreducible representations of $\mathcal J(G)$. As a consequence, we prove a conjecture of Bezrukavnikov-Braverman-Kazhdan that the inclusion $\mathcal H(G)\subset\mathcal J(G)$ induces an isomorphism $\mathcal H(G)/[\mathcal H(G),\mathcal H(G)]\simeq\mathcal J(G)/[\mathcal J(G),\mathcal J(G)]$ on the cocenters. We also provide an explicit description of $\mathcal J(G)$ and the cocenter $\mathcal H(G)/[\mathcal H(G),\mathcal H(G)]$ when $\mathbf G=\mathrm{GL}_n$.
\end{abstract}

\maketitle

\section{Introduction}

Let $\mathbf G$ be a reductive algebraic group over a non-archimedean local field $F$ of characteristic zero and let $G=\mathbf G(F)$ be the group of $F$-rational points (throughout the paper, bold-faced letters such as $\mathbf G$, $\mathbf P$, and $\mathbf M$ will denote algebraic groups over $F$ and normal letters such as $G$, $P$, and $M$ will denote their $F$-rational points). Let $\cH(G)$ be the \emph{Hecke algebra}, the space of locally constant measures on $G$ with compact support.

For an Iwahori subgroup $I\subset G$ the subalgebra $\cH(G,I)\subset \cH(G)$ of $I\times I$-invariant measures is an affine Hecke algebra with parameter $q_F$, the size of the residue field of $F$. In \cite{cells2} Lusztig defines the \emph{asymptotic Hecke algebra} $J$, which is the ``limit" of $\cH(G,I)$ as $q_F\to\infty$. Although $J$ is defined using the canonical basis, \cite{braverman-kazhdan,geometric-realization} prove a spectral description of $J$. Braverman and Kahzdan in turn use the spectral definition to extend $J$ and define the full asymptotic Hecke algebra $\cJ(G)$ which recovers Lusztig's algebra upon taking $I$-invariants:
\[
\cJ(G,I)\colonequals\cJ(G)^{I\times I}\simeq J.
\]
In \S\ref{sec:gln} we provide an explicit description of $\mathcal J(G)$ when $\mathbf G=\GL_n$ using the explicit description of the Hecke algebra in \cite{bushnell-kutzko}.

In \cite{bezrukavnikov-braverman-kazhdan} Bezrukavnikov, Braverman, and Kazhdan proves that the inclusion $\cH(G,I)\subset \cJ(G,I)$ induces an isomorphism on the cocenters\[\cH(G,I)/[\cH(G,I),\cH(G,I)]\simeq \cJ(G,I)/[\cJ(G,I),\cJ(G,I)]\] and conjectures the isomorphism extends to the full Hecke algebra. We prove this conjecture:
\begin{thm}[{\cite[Conjecture~1]{bezrukavnikov-braverman-kazhdan}}]\label{main-thm}
    Let $\cH(G)$ be the Hecke algebra, and let $\cJ(G)$ be Braverman and Kahzdan's asymptotic Hecke algebra. The inclusion $\cH(G)\subset \cJ(G)$ induces an isomorphism
    \begin{equation}\label{eq:main-theorem-intro}
    \cH(G)/[\cH(G),\cH(G)]\simeq\cJ(G)/[\cJ(G),\cJ(G)].
    \end{equation}
\end{thm}

\begin{rmk}
    In \cite{solleveld-HH}, Solleveld computes the higher Hochschild homology of $\cH(G)$ and $\cC(G)$, the ring of Schwartz functions. It may be interesting to compute the higher Hochschild homology of $\cJ(G)$ and compare the results.
\end{rmk}

\begin{rmk}
    In \cite[Theorem~E]{solleveld-HH} Solleveld observes that the Hochschild homology of $\cH(G)$ decomposes along Harish-Chandra blocks, not just Bernstein blocks. Theorem~\ref{main-thm} gives a conceptual reason why $HH_0(\cH(G))$ admits a better decomposition than expected: although $\cH(G)$ does not decompose along Harish-Chandra blocks, $\cJ(G)$ does.
\end{rmk}

In \cite{bezrukavnikov-braverman-kazhdan} the proof of the isomorphism~\eqref{eq:main-theorem-intro} at Iwahori level uses the canonical basis and the cell filtration on $\cH(G,I)$, neither of which are available for the full Hecke algebra $\cH(G)$. Instead, we use the decomposition~\eqref{eq:direct-sum-decomp-C} in terms of Harish-Chandra blocks.

To prove Theorem~\ref{main-thm}, we first classify irreducible modules of $\cJ(G)$:
\begin{prop}[{Theorem~\ref{thm:classification-of-irreducibles}}]
    The irreducible modules of $\cJ(G)$ are exactly the representations $I_P^G(\sigma\otimes\chi)$, where $\mathbf P=\mathbf M\ltimes \mathbf U\subset \mathbf G$ is a parabolic subgroup, $\chi\colon M\to \C^\times$ is a strictly positive unramified character, and $\sigma$ is an irreducible tempered representation of $M$.
\end{prop}
From the proposition we can prove a trace Paley-Wiener theorem (Theorem~\ref{thm:trace-pw-J}) for $\cJ(G)$. Combining the trace Paley-Wiener theorem for $\cJ(G)$ with the classification of irreducible $\cJ(G)$-modules and the Langlands classification, we obtain a diagram of isomorphisms
\[ \begin{tikzcd}
\mathcal H/[\mathcal H,\mathcal H] \arrow{r} \arrow{d}{\simeq} & \mathcal J/[\mathcal J,\mathcal J] \arrow{d}{\simeq} \\
\Hom_{\reg}(R(\mathcal H),\C) \arrow{r}{\simeq}& \Hom_{\reg}(R(\mathcal J),\C).
\end{tikzcd}
\]
Thus the top homomorphism must also be an isomorphism, proving Theorem~\ref{main-thm}.

In \S\ref{sec:prelim} we review the definition of Braverman and Kazhdan's asymptotic Hecke algebra, Bernstein and Harish-Chandra blocks, and the theory of intertwining operators. In \S\ref{sec:gln} we explicitly compute $\cJ(G)$ when $\bG=\GL_n$, using the description of Hecke algebras of Bernstein components in \cite{bushnell-kutzko}, and use it to explicitly describe $\cH(G)/[\cH(G),\cH(G)]$. In \S\ref{sec:classifying-irreps} we go over basic ring-theoretic properties of $\cJ(G)$, and we classify irreducible representations of $\cJ(G)$. In \S\ref{sec:trace-pw} we prove a trace Paley-Wiener theorem for $\cJ(G)$ and as a consequence we prove Theorem~\ref{thm:classification-of-irreducibles}.

\subsection*{Acknowledgements}
The author thanks Roman Bezrukavnikov for suggesting the problem and his numerous helpful suggestions. The author thanks Vasily Krylov for pointing out an error in the original statement and proof of Proposition~\ref{prop:modulo-m}. The author also thanks Stefan Dawydiak, Do Kien Hoang, Ju-Lee Kim, Ivan Losev, and David Vogan for helpful discussions.

\section{Preliminaries}\label{sec:prelim}

\subsection{Braverman-Kazhdan's asymptotic Hecke algebra}

We will summarize the construction of $\cJ(G)$ from \cite{braverman-kazhdan}. Let $\Rep(G)$ be the category of smooth representations of $G$ over $\C$ and let $\Rep_t(G)$ be the subcategory of tempered representations of $G$.

For a reductive group $\mathbf G$ over $F$, let $\Psi(G)$ be the group of unramified characters of $G$, and let $\Psi_u(G)$ be the group of unitary unramified characters of $G$. Note that $\Psi(G)$ is a complex torus and can be viewed as an algebraic variety, and $\Psi_u(G)$ is the maximal compact subgroup.

\subsubsection{Matrix Paley-Wiener theorems}
There are versions of matrix Paley-Wiener theorems for both smooth representations and tempered representations. Observe that for a ($F$-rational) parabolic subgroup $\mathbf P=\mathbf M\ltimes\mathbf U\subset \mathbf G$ and a smooth representation $\sigma\in \Rep(M)$, for all unramified characters $\chi\in\Psi(M)$ the representations $I_P^G(\sigma\otimes \chi)$ can all be realized on the same vector space $I_P^GV_{[\sigma]}=I_{P\cap K_0}^{K_0}\sigma$ where $K_0\subset G$ is a hyperspecial maximal compact subgroup. The Hecke algebra $\cH(G)$ admits a spectral description:
\begin{lemma}[{\cite[Theorem~25]{draft}}]
    There is an isomorphism between the Hecke algebra $\cH(G)$ and the endomorphisms $\eta$ of the forgetful functor $\Forg\colon\Rep(G)\to\Vect_\C$ such that:
    \begin{itemize}
        \item there exists a compact open subgroup $K$ of $G$ such that $\eta$ is $K\times K$-invariant; and 
        \item for all parabolic subgroups $\mathbf P=\mathbf M\ltimes\mathbf U\subset \mathbf G$ and all representations $\sigma\in \Rep(M)$ and $\chi\in\Psi(M)$ the endomorphisms $\eta_{I_P^G(\sigma\otimes\chi)}\in\End_\C(I_P^GV_{[\sigma]})$ are regular functions of $\chi$.
    \end{itemize}
\end{lemma}
There is an analogous statement for tempered representations \cite{waldspurger}:
\begin{lemma}
    There is an isomorphism between the Schwartz algebra $\cC(G)$ and the endomorphisms $\eta$ of the forgetful functor $\Forg_t\colon\Rep_t(G)\to\Vect_\C$ such that:
    \begin{itemize}
        \item there exists a compact open subgroup $K$ of $G$ such that $\eta$ is $K\times K$-invariant; and 
        \item for all parabolic subgroups $\mathbf P=\mathbf M\ltimes\mathbf U\subset \mathbf G$ and all tempered representations $\sigma\in \Rep_t(M)$ and $\chi\in\Psi_u(M)$ the endomorphisms $\eta_{I_P^G(\sigma\otimes\chi)}\in\End_\C(I_P^GV_{[\sigma]})$ are smooth functions\footnote{In the real analytic sense; recall that $\Psi_u(M)$ is a compact real analytic manifold.} of $\chi$.
    \end{itemize}
\end{lemma}
The asymptotic Hecke algebra can be defined analogously, sitting between the two rings above: $\cH(G)\subset \cJ(G)\subset \cC(G)$:
\begin{defn}
    Let $\cJ(G)$ be the subring of $\cC(G)$ consisting of endomorphisms $\eta$ of the forgetful functor $\Forg_t\colon \Rep_t(G)\to \Vect_\C$ such that for all parabolic subgroups $\mathbf P=\mathbf M\ltimes \mathbf U\subset\mathbf G$ and all tempered representations $\sigma\in\Rep_t(G)$ the smooth function \begin{align*}
        \Psi_u(M)&\to \End(I_P^GV_{[\sigma]})\\
        \chi&\mapsto\eta_{I_P^G(\sigma\otimes\chi)}
    \end{align*}
    extends to a rational function on $\Psi(M)$ which is regular on the (non-strictly) positive characters $\chi\in X(M)$.
\end{defn}
Given a compact open subgroup $K\subset G$ and $*\in\{\cH,\cJ,\cC\}$ let $*(G,K)\subset *(G)$ be the subring of $K\times K$-invariant functions, and let $\Rep(G,K)\subset\Rep(G)$ be the subcategory of representations generated by its $K$-invariant vectors and let $\Rep_t(G,K)=\Rep(G,K)\cap\Rep_t(G)$. We have \begin{equation}\label{eq:colimit}*(G)=\lim_{\longrightarrow K}*(G,K)\hspace{0.3cm}\text{for }*\in\{\cH,\cJ,\cC\},\end{equation}
so most of the time we can focus our attention on these families of subrings. Then $\cH(G,K)$ is a subring of endomorphisms of the forgetful functor $\Rep(G,K)\to\Vect_\C$ and $\cJ(G,K)$ and $\cC(G,K)$ is a subring of the endomorphsims of the forgetful functor $\Rep_t(G,K)\to \Vect_\C$.
\begin{example}
    When $K=I$ is the Iwahori subgroup, \cite{braverman-kazhdan,geometric-realization} proves that $\cJ(G,I)$ matches the asymptotic Hecke algebra defined by Lusztig in \cite{cells2}.
\end{example}

\subsection{Bernstein decomposition}
The Hecke algebra $\cH(G)$ can be viewed as a sheaf of algebras over $\Theta(G)=\Spec \mathfrak Z(G)$, where $\mathfrak Z(G)$ is the Bernstein center \cite{bernstein-center}. Recall that $\Theta(G)$ decomposes as a disjoint union
\[
\Theta(G)=\bigsqcup_\fs \Theta^\fs(G),
\]
where $\fs$ runs over $G$-conjugacy classes of pairs $[\mathbf M,\sigma]$ where $\mathbf M\subset\mathbf G$ is a Levi subgroup and $\sigma$ is an irreducible supercuspidal representation of $M$, up to twisting by an unramified character of $M$. This induces decompositions:
\begin{equation}\label{eq:direct-sum-decomp}
\cH(G)=\bigoplus_\fs\cH(G)^\fs,\hspace{0.5cm}
\cJ(G)=\bigoplus_\fs\cJ(G)^\fs.
\end{equation}
\begin{example}
    When $\mathbf G$ has a split maximal torus $\mathbf T$ and $[\mathbf T,\triv_T]$ then $\cH(G)^\fs=\cH(G,I)$ is the Iwahori Hecke algebra and $\cJ(G)^\fs=\cJ(G,I)$ is isomorphic to Lusztig's asymptotic Hecke algebra \cite{braverman-kazhdan}.
\end{example}
A completely analogous decomposition, into \emph{Harish-Chandra blocks},\footnote{We use the terminology of \cite[page~6]{solleveld-HH}.} proved in \cite{waldspurger}, we have a decomposition of the category of tempered representations of $G$:
\[
\Rep_t(G)=\bigoplus_{\mathfrak d} \Rep_t^{\mathfrak d}(G),
\]
where $\mathfrak d$ runs over $G$-conjugacy classes of pairs $[\mathbf M,\delta]$ where $\mathbf M\subset \mathbf G$ is a Levi subgroup and $\delta$ is a square-integrable representation of $M$. They give rise to decompositions
\begin{equation}\label{eq:direct-sum-decomp-C}
\cC(G)=\bigoplus_\fd\cC(G)^\fd,\hspace{0.5cm}
\cJ(G)=\bigoplus_\fd\cJ(G)^\fd.
\end{equation}
\begin{rmk}\label{rmk:refinement}
    The decomposition of $\cJ$ in \eqref{eq:direct-sum-decomp-C} is more refined than in \eqref{eq:direct-sum-decomp}; in particular, for a Bernstein block $\fs$ there is a finite set $\Delta_G^\fs$ of Harish-Chandra blocks such that
    \[
    \cJ(G)^\fs=\bigoplus_{\fd\in\Delta_G^\fs}\cJ(G)^\fd.
    \]
\end{rmk}
\begin{rmk}
    By \cite{some-examples} when $\bG$ is split over $F$, the Harish-Chandra blocks are parameterized by a Levi $\bM\subset\bG$ with an admissible triple $(s,u,\rho)$ such that $su=us$ where $s\in \bG^\vee(\C)$ is compact semisimple and $u\in\bG^\vee(\C)$ is unipotent, and $\Cent_{\bG^\vee}(s,u)$ does not contain a nontrivial torus. Thus in particular \eqref{eq:direct-sum-decomp-C} is also a refinement of the decomposition of $\cJ(G,I)$ by two-sided cells.
\end{rmk}

\subsection{Intertwining operators}

The asymptotic Hecke algebra can be understood via intertwining operators from \cite{waldspurger}.

\begin{defn}
    For each Levi subgroup $\bM\subset \bG$, let $\cP(\bM)$ be the set of parabolic subgroups $\bP$ of $\bG$ which have $\bM$ as its Levi component.
\end{defn}
Now, fix a Levi subgroup $\bM\subset \bG$ and a $\Psi(M)$-orbit $[\pi]$ in the set $\Rep(M)/\simeq$ of finite length representations of $M$ modulo isomorphism. For any representative $\pi\in[\pi]$ the orbit $[\pi]=\Psi(M)\pi$ is in bijection with $\Psi(M)/\Psi(M,\pi)$ where
\[
\Psi(M,\pi)\colonequals\{\chi\in\Psi(M):\pi\otimes\chi\simeq\pi\},
\]
is a finite group so $[\pi]$ may be viewed as an algebraic variety. Although the bijection non-canonically depends on the choice of $\pi\in[\pi]$ neither the subgroup $\Psi(M,\pi)\subset \Psi(M)$ nor the algebraic structure on $[\pi]$ depends on the choice. All the representations $\pi\in[\pi]$ (resp., the parabolic inductions $I_P^G\pi$) can be realized on a single vector space, which we denote by $V_{[\pi]}$ (resp., $I_P^GV_{[\pi]}$.) They may be viewed as vector bundles of $M$-modules (resp., $G$-modules) on the algebraic variety $[\pi]$.

Let $B_{[\pi]}\colonequals\Frac\cO([\pi])$. For any two $\bP,\bP'\in \cP(\bM)$, Waldspurger \cite[Th\'eor\`eme~IV.1.1]{waldspurger} defines intertwining operators
\[J_{P'|P}\in B_{[\pi]}\otimes \Hom_G(I_P^GV_{[\pi]},I_{P'}^GV_{[\pi]}),\] i.e., $G$-equivariant homomorphisms $I_P^GV_{[\pi]}\to I_{P'}^GV_{[\pi]}$ defined on a Zariski open subset of $[\pi]$. 
\begin{rmk}\label{rmk:scalar-independence}
    By Schur's lemma $B_{[\pi]}\otimes \Hom_G(I_P^GV_{[\pi]},I_{P'}^GV_{[\pi]})$ is a one-dimensional $B_{[\pi]}$-vector space, so any choice of intertwining operators only differs by a scalar in $B_{[\pi]}$. Thus, we may canonically define a $G\times G$-equivariant map
    \[
B_{[\pi]}\otimes \End_{\sm}(I_P^GV_{[\pi]})\to B_{[\pi]}\otimes \End_{\sm}(I_{P'}^GV_{[\pi]})
\]
where
\[\End_{\sm}I_P^GV_{[\pi]}\colonequals I_P^GV_{[\pi]}\otimes I_P^GV^\vee_{[\pi]}=\lim_{\longrightarrow K}\End_\C\big((I_P^GV_{[\pi]})^K\big)\] denotes the ring of all smooth endomorphisms. Thus, for most of our arguments, the particular choice $J_{P'|P}$ does not play a role.
\end{rmk}
The following provides key control over the location of the poles of intertwiners:
\begin{lemma}[{\cite[Lemme~V.3.1]{waldspurger}}]\label{lem:regular}
    For any Harish-Chandra block $\fd=[\bM,\delta]$ of $\bG$ and $\bP,\bP'\in\cP(\bM)$ there is an intertwining operator in $B_{[\delta]}\otimes\Hom_G(I_P^GV_{[\delta]},I_{P'}^GV_{[\delta]})$ which give unitary isomorphisms on $\Psi_u(M)\delta\subset\Psi(M)\delta$.
\end{lemma}
\begin{rmk}
    In particular, when $\delta$ is a square-integrable representation, the isomorphism type of the parabolic induction $I_P^G\delta$ is independent of the choice of parabolic subgroup $\bP\in\cP(\bM)$.
\end{rmk}
As a special case of the above constructions, for $w$ in the relative Weyl group $W_\bM\colonequals N_\bG(\bM)/\bM$, we have intertwining operators $J_w$ which are rational homomorphisms
\[
I_P^G(\delta\otimes\chi)\xrightarrow{J_{w^{-1}P|P}} I_{w^{-1}\cdot P}^G(\delta\otimes\chi)\simeq I_P^G(\delta^w\otimes\chi^w).
\]
For example, when $w\in W_{[\delta]}$, the stabilizer of the $W_\bM$-action on $[\delta]\in\Rep(M)/\simeq$, so that $\delta^w\simeq\delta\otimes \chi'$, we have rational intertwining operators
\[
J_w\colon I_P^G(\delta\otimes\chi)\to I_P^G(\delta\otimes\chi^w\chi').
\]
By Lemma~\ref{lem:regular}, when $\delta$ is square-integrable the $J_w$ is regular when $\chi\in\Psi_u(M)$. In this situation, we have:
\begin{prop}[{\cite[Theorem~5.5.3.2]{silberger-book}}]\label{prop:commuting-alg-theorem}
    For a Harish-Chandra block $\fd=[\bM,\delta]$ and unitary unramified characters $\chi,\chi'\in\Psi(M)$, the $G$-equivariant homomorphisms $I_P^G(\delta\otimes\chi)\to I_P^G(\delta\otimes\chi')$ are spanned by intertwining operators $I_w$ such that $\delta^w\otimes\chi^w\simeq \delta\otimes\chi'$.
\end{prop}
\begin{rmk}
    In particular, the representations $I_P^G(\delta\otimes\chi)$ are semisimple.
\end{rmk}

Thus when $\delta$ is square-integrable the intertwining operators define a projective unitary representation of $W_{[\delta]}$ on $I_P^G(\delta)$ and the irreducible components are exactly the direct summands of $I_P^G(\delta)$ as a $G$-representation.

Now although operators are only determined up to constants, but by Remark~\ref{rmk:scalar-independence} and Lemma~\ref{lem:regular}, they define well-defined homomorphisms
\[C^\infty(\Psi_u(M)\delta)\otimes\End_\sm(I_P^G(\delta\otimes\chi))\simeq C^\infty(\Psi_u(M)\delta)\otimes\End_\sm(I_P^G(\delta\otimes\chi)).\]
Proposition~\ref{prop:commuting-alg-theorem} gives the Plancherel isomorphism \cite{waldspurger}:
\begin{prop}\label{prop:plancherel}
    Let $\fd=[\bM,\delta]$ be a Harish-Chandra block. Then
    \[
    \cC(G)^\fd\simeq \big(C^\infty(\Psi_u(M)\delta)\otimes\End_\sm(I_P^GV_{[\delta]})\big)^{W_{[\delta]}}.
    \]
    In particular, for any compact open subgroup $K$,
    \[
    \cC(G,K)^\fd\simeq \big(C^\infty(\Psi_u(M)\delta)\otimes \End_\C((I_P^GV_{[\delta]})^K)\big)^{W_{[\delta]}}.
    \]
\end{prop}
\section{Explicit description of $\cJ(G)$ when $\mathbf G=\GL_n$}\label{sec:gln}

Recall that the asymptotic Hecke algebra $\cJ(G)$ decomposes into a direct sum of $\cJ(G)^\fs$ where $\fs$ are Bernstein blocks. Bernstein blocks of $G$ are of the form:
\begin{equation}\label{eq:bernstein-block}
\fs=[\GL_{n_1}^{r_1}\times\GL_{n_2}^{r_2}\times\cdots\times\GL_{n_k}^{r_k},\sigma_1^{\boxtimes r_1}\boxtimes\sigma_2^{\boxtimes r_2}\boxtimes\cdots\boxtimes \sigma_k^{\boxtimes r_k}],
\end{equation}
where the representations $\sigma_i$ of $\GL_{n_i}$ are supercuspidal and not unramified twists of each other. Let $f_i$ denote the \emph{torsion number} of $\sigma_i$, i.e., the order of the finite cyclic group of unramified characters
\[
\{\eta\in \Psi(\GL_{n_i}(F)):\sigma_i\otimes\eta\simeq\sigma_i\}.
\]Then, the main theorem of \cite{bushnell-kutzko} (see also \cite{karemaker-hecke}) states:
\begin{prop}[{\cite[Theorem~1.1]{bushnell-kutzko}, \cite[\S8.4]{bushnell-kutzko-types}}]\label{prop:BK-iso}
    Let $\bG=\GL_n$ and let $\fs$ be a Bernstein block, as in \eqref{eq:bernstein-block}. Then there is an isomorphism
    \begin{equation}\label{eq:BK-iso}
    \cH(G)^\fs\simeq\bigotimes_{i=1}^k\cH(r_i,q_F^{f_i})
    \end{equation}
    where $\cH(r_i,q_F^{f_i})$ is the affine Hecke algebra of $\GL_{r_i}$ with parameter $q_F^{f_i}$, and $q_F$ is the size of the residue field of $F$. Moreover, these isomorphisms are compatible with twisting, parabolic induction, and preserving tempered representations. 
\end{prop}
In particular, let us identify unramified characters of $\GL_{n_i}(F)$ with $\C^\times$ under the isomorphism:
\[
    \C^\times\simeq \Psi(\GL_{n_i}(F)):
    \alpha_i\mapsto (g\mapsto \alpha_i^{v_F(\det g)}).
\]
Then using notation from \cite{BZ2}, the induced representation, given $\alpha_{i1},\dots,\alpha_{ir_i}\in\C^\times$, \eqref{eq:BK-iso} takes the $\bigotimes_{i=1}^k\cH(r_i,q_F^{f_i})$-module
\[
\bigotimes_{i=1}^kI_{B_{r_i}}^{\GL_{r_i}}(\alpha_{i1}^{f_i},\dots,\alpha_{ir_i}^{f_i})
\]
to
\[
(\alpha_{11}\sigma_1\times\cdots\times\alpha_{1r_1}\sigma_1)\times\cdots\times (\alpha_{k1}\sigma_k\times\cdots\times\alpha_{kr_k}\sigma_k).
\]
In particular, \eqref{eq:BK-iso} preserves Zelevinky's classification of irreducible $\GL_n(F)$-representations. 
\begin{rmk}
    Since Zelevinsky's classification is compactible with the local Langlands correspondence for $\GL_n$, the isomorphism~\eqref{eq:BK-iso} preserves Langlands parameters, in the following sense. Let $\varphi_{\sigma_i}\colon W_F\to \GL_{n_i}(\C)$ be the Langlands parameter of $\pi_i$. The Langlands parameter of an irreducible representation $\pi$ of $\GL_n$ in the Bernstein block $\fs$ is of the form
    \[
    \varphi_1\otimes\rho_1\oplus\varphi_2\otimes\rho_2\oplus\cdots\oplus\varphi_k\otimes\rho_k,
    \]
    where $\rho_i\colon W_F\ltimes\C\to \GL_{r_i}(\C)$ are unramified Langlands parameters. Recall that the data of $\rho_i$ is equivalent to a semisimple $s_i\in\GL_{r_i}(\C)$ together with a unipotent $u\in \GL_{r_i}(\C)$ such that $s_iu_is_i^{-1}=u_i^{q_F}$. Now $s_i^{f_i}u_is_i^{-f_i}=u_i^{q_F^{f_i}}$ so the tuple $(s_i^{f_i},u_i)$ corresponds to a representation $\tau_i$ of $\cH(r_i,q_F^{f_i})$. Under the isomorphism \eqref{eq:BK-iso} the representation $\pi$ corresponds to $\bigotimes_{i=1}^k\tau_i$.
\end{rmk}
Since the parabolic inductions are how we define the asymptotic Hecke algebra, we obtain:
\begin{thm}\label{prop:gln-J-explicit}
     Let $\bG=\GL_n$ and let $\fs$ be a Bernstein block, as in \eqref{eq:bernstein-block}. Then there is an isomorphism
     \begin{equation}\label{eq:gln-J-explicit}
     \cJ(G)^\fs\simeq\bigotimes_{i=1}^k\cJ(r_i),
     \end{equation}
     where $\cJ(r_i)$ are the asymptotic Hecke algebras of $\GL_{r_i}$.
\end{thm}
\begin{rmk}
The isomorphism \eqref{eq:gln-J-explicit} also preserves Harish-Chandra blocks. In the notation of Remark~\ref{rmk:refinement}, for $\fs$ as in \eqref{eq:bernstein-block}, the set $\Delta_G^\fs$ of Harish-Chandra blocks $\fd_{\lambda_1,\dots,\lambda_k}$ is parameterized by partitions $\lambda_i=[a_{i1},\dots,a_{i\ell}]$ of $r_i$, obtained by replacing each $[\GL_{n_i}^{r_i},\sigma_i^{\boxtimes r_i}]$ with
\[
[\GL_{n_ia_{i1}}\times\cdots\times \GL_{n_ia_{i\ell}},\mathrm{St}_{a_{i1}}(\sigma_i)\boxtimes\cdots\boxtimes \mathrm{St}_{a_{i\ell}}(\sigma_i)],
\]
where $\mathrm{St}_m(\sigma)$ denotes generalized Steinberg representations. Now \eqref{eq:gln-J-explicit} restricts to:
\[
\cJ(G)^{\fd_{\lambda_1,\dots,\lambda_k}}\simeq\bigotimes_{i=1}^k\cJ(r_i)_{c_{\lambda_i}},
\]
where $c_{\lambda_i}$ is the two-sided cell corresponding to the partition $\lambda_i$ of $r_i$.
\end{rmk}

Our explicit description gives a proof of Theorem~\ref{main-thm} when $\bG=\GL_n$, by the following easy lemma, that taking cocenters commutes with tensor products:
\begin{lemma}\label{lem:HH-commute-tensor}
    Let $A$ and $B$ be unital $\C$-algebras. Then
    \[
    A\otimes B/[A\otimes B,A\otimes B]\simeq A/[A,A]\otimes B/[B,B].
    \]
\end{lemma}
\begin{proof}
    It suffices to check
    \begin{equation}\label{eq:commutators}
    [A\otimes B,A\otimes B]=A\otimes [B,B]+[A,A]\otimes B.
    \end{equation}
    The inclusion $\subseteq$ holds since for $a_1,a_2\in A$ and $b_1,b_2\in B$,
    \begin{align*}
    [a_1\otimes b_1,a_2\otimes b_2]&=a_1a_2\otimes b_1b_2-a_2a_1\otimes b_2b_1\\
    &=a_1a_2\otimes b_1b_2-a_2a_1\otimes b_1b_2+a_2a_1\otimes b_1b_2-a_2a_1\otimes b_2b_1\\
    &=[a_1,a_2]\otimes b_1b_2+a_2a_1\otimes[b_1,b_2]
    \end{align*}
    which is in $[A,A]\otimes B+A\otimes[B,B]$. Conversely, $[A,A]\otimes B\subset [A\otimes B,A\otimes B]$ since for $a_1,a_2\in A$ and $b\in B$,
    \[
    [a_1,a_2]\otimes b=[a_1\otimes b,a_2\otimes 1]\in [A\otimes B,A\otimes B].
    \]
    Similarly $A\otimes [B,B]\subset [A\otimes B,A\otimes B]$, which concludes the proof of \eqref{eq:commutators}.
\end{proof}
\begin{proof}[{Proof of Theorem~\ref{main-thm} when $\bG=\GL_n$}]
    By \eqref{eq:direct-sum-decomp} it suffices to check that for any Bernstein block $\fs$ of $\GL_n$ of the form \eqref{eq:bernstein-block},
    \[
    \cH(G)^\fs/[\cH(G)^\fs,\cH(G)^\fs]\simeq \cJ(G)^\fs/[\cJ(G)^\fs,\cJ(G)^\fs].
    \]
    Now by Proposition~\ref{prop:BK-iso} and Theorem~\ref{prop:gln-J-explicit}, together with Lemma~\ref{lem:HH-commute-tensor} we are reduced to checking
    \[
    \cH(r_i,q_F^{f_i})/[\cH(r_i,q_F^{f_i}),\cH(r_i,q_F^{f_i})]\simeq \cJ(r_i)/[\cJ(r_i),\cJ(r_i)].
    \]
    This is proved in \cite{bezrukavnikov-braverman-kazhdan}.
\end{proof}
In fact, we can provide an explicit description of the cocenter of the Hecke algebra:
\begin{prop}
For each Harish-Chandra block $\fd=[\bM,\delta]$ of $\GL_n$, we have
\[
\cJ(G)^\fd/[\cJ(G)^\fd,\cJ(G)^\fd]\simeq \cO(\Psi(M)\delta)^{W_{[\delta]}},
\]
where $W_{[\delta]}$ is the stabilizer of $N_\bG(\bM)/\bM$ acting on the inertial equivalence class $\Psi(M)\delta$ of $\delta\in\Irr_t(M)$. In particular, for each Bernstein block $\fs$ as in \eqref{eq:bernstein-block},
\[
\cH(G)^\fs/[\cH(G)^\fs,\cH(G)^\fs]\simeq \bigotimes_{i=1}^k\bigg(\bigoplus_{\lambda_1+\cdots+\lambda_k=n}\cO(\mathbb G_m^k)^{H_{\lambda}}\bigg),
\]
where $H_\lambda\subset S_k$ is the stabilizer of $(\lambda_1,\dots,\lambda_k)$.
\end{prop}
\begin{proof}
By Theorem~\ref{prop:gln-J-explicit} it suffices to describe the cocenter of the asymptotic Hecke algebra $\cJ(n)$. Recall that each asymptotic Hecke algebra $\cJ(n)$ is decomposed into a direct sum
\[
\cJ(n)=\bigoplus_e\cJ(n)_e\]
running over nilpotent orbits of $\GL_n$, i.e., partitions $\lambda=[n_1,\dots,n_k]$ of $n$. By the same argument as \cite[Remark~2.5.3]{geometric-realization} each $\cJ(n)_e$ are just matrix algebras over its center, hence the cocenter is isomorphic to the center. The center is Corollary~\ref{cor:lattice}, but also follows from the description $\Cent(\cJ(n)_e)\simeq K_{\Cent_e}(*)$ in \cite{geometric-realization}. The description of the cocenter of $\cH(G)^\fs$ follows immediately.
\end{proof}

\section{Basic properties of Braverman and Kazhdan's asymptotic Hecke algebra}\label{sec:classifying-irreps}

\subsection{Finiteness properties}
Let $\fd=[\bM,\delta]$ be a Harish-Chandra block and fix a parabolic $\bP\in\cP(\bM)$. By definition the image of $\cJ(G)^\fd\subset \cC(G)^\fd$ under this isomorphism lands in 
\[
\big(C^{rat}(\Psi_u(M)\delta)\otimes\End_\sm(I_P^GV_{[\delta]})\big)^{W_{[\delta]}},
\]
where $\End_{\sm}$ is defined in Remark~\ref{rmk:scalar-independence} and
\[C^{rat}(\Psi_u(M))\colonequals\{f\in\Frac\mathcal O(\Psi(M)\delta):f\text{ has no poles on }\Psi_u(M)\delta\}\]
is the subring of $C^\infty(\Psi_u(M)\delta)$ which extends to a rational function on $\Psi(M)\delta$. Here again by the bijection $\Psi(M)/\Psi_u(M,\delta)\simeq \Psi(M)\delta$ the set $\Psi(M)\delta$ can be viewed as an algebraic variety over $\C$. In particular, we obtain an embedding
\[
\cJ(G,K)^\fd\hookrightarrow \big(B_\fd\otimes \End_\sm(I_P^GV_{[\delta]})\big)^{W_{[\delta]}}
\]
where we denote $B_\fd\colonequals\Frac\mathcal O(\Psi(M)\delta)$. In fact, we can characterize the image of the embedding:
\begin{prop}\label{prop:plancherel-j}
    Let $\fd=[\bM,\delta]$ be a Harish-Chandra block. Then there is an isomorphism of $\cJ(G,K)^\fd$ with all rational functions $\eta\colon \Psi(M)\delta\to \End_\sm(I_P^GV_{[\delta]})$ such that:
    \begin{enumerate}
        \item\label{item-equivariance} $\eta$ is $W_{[\delta]}$-equivariant; and
        \item\label{item-poles} for any $\bP'\in\cP(\bM)$ the rational endomorphism $I_{P'|P}\circ \eta\circ I_{P'|P}^{-1}$ of $B_\fd\otimes \End_\sm(I_{P'}^GV_{[\delta]})$ is regular for $\chi$ non-strictly positive with respect to $\bP'$.
    \end{enumerate}
\end{prop}
As a corollary to Proposition~\ref{prop:plancherel-j}, we obtain a finite generation result:
\begin{cor}\label{cor:lattice} For a Harish-Chandra block $\fd=[\bM,\delta]$ the center of $\cJ(G)^\fd$ is  $\cO(\Psi(M)\delta)^{W_{[\delta]}}$. Moreover, for any compact open subgroup $K$ the ring $\cJ(G,K)^\fd$ is a $\cO(\Psi(M)\delta)^{W_{[\delta]}}$-lattice in $\big(B_\fd\otimes \End_\C(I_P^G(V_{[\delta]})^K)\big)^{W_{[\delta]}}$.
\end{cor}
\begin{proof}
    The inclusion $\cO(\Psi(M)\delta)^{W_{[\delta]}}\subset \Cent(\cJ(G)^\fd)$ is clear, for example from Proposition~\ref{prop:plancherel}. Conversely, it is clear $\Cent(\cJ(G)^\fd)\subset B_\fd^{W_{[\delta]}}$. Now, from Proposition~\ref{prop:plancherel-j} we know any central element must be regular everywhere, since any character $\chi\in\Psi(M)$ is positive with respect to some $\bP'\in\cP(\bM)$, so actually $\Cent(\cJ(G))^\fd\subset \cO(\Psi(M)\delta)^{W_{[\delta]}}$.

    For the next statement, it suffices to note the inclusions:
    \[
    (\cO(\Psi(M)\delta)\otimes\End_\C(I_P^G(V_{[\delta]})^K)^{W_{[\delta]}}\subset \cJ(G,K)^\fd\subset \sum_{\bP'\in\cP(\bM)}J_{P'|P}^{-1}(\cO(\Psi(M)\delta)\otimes\End_\C(I_P^G(V_{[\delta]})^K)^{W_{[\delta]}}J_{P'|P}.
    \]
\end{proof}
\begin{rmk}
    For the Iwahori block and for the lowest two-sided cell, the description of the center in Corollary~\ref{cor:lattice}  is \cite[Theorem~3.6]{lowest-cell}.
\end{rmk}

Since $\cJ(G,K)^\fd$ is finitely geneated, it plays well with completion. 

\subsection{Classifying irreducible representations of the asymptotic Hecke algebra} Now, $\cJ(G,K)^\fd$ is a sheaf over the GIT quotient $\Psi(M)\delta/\!/W_{[\delta]}$, and we can describe its fibers.

\begin{prop}\label{prop:modulo-m}
    Let $\fd=[\bM,\delta]$ be a Harish-Chandra block, let $K\subset G$ be a compact open subgroup. For any $\chi\delta\in \Psi(M)\delta$, let $\mathfrak m_{W_{[\delta]}\chi}$ be the corresponding maximal ideal of $\Cent(\cJ(G,K)^\fd)=\cO(\Psi(M)\delta)^{W_{[\delta]}}$. Then $\cJ(G,K)^\fd/\mathfrak m_{W_{[\delta]}\chi}$ has irreducible representations $I_Q^G(\pi\otimes\chi')^K$, where $\mathbf Q=\mathbf L\ltimes\mathbf U'\subset \bG$ is a parabolic subgroup containing $\bM$ such that $\chi'$ is strictly positive with respect to $\mathbf Q$ and $\pi\in\Irr_t(L)^{[\bM,\delta]}$, so that $\pi\otimes\chi'$ is a direct summand of $I_{M}^L(\delta\otimes\chi)$. Moreover, the surjection
    \[
    \cJ(G,K)^\fd/\mathfrak m_{W_{[\delta]}\chi}\to \bigoplus \End_\C(I_Q^G(\pi\otimes\chi')^K)
    \]
    induces an isomorphism on the cocenters.
\end{prop}
\begin{proof}
First let us consider the completion $\cJ^\land$ of $\cJ(G,K)^\fd$ at $\mathfrak m_{W_{[\delta]}\chi}$.

    Fix a parabolic $\bP\in\cP(\bM)$ such that $\chi$ is non-strictly positive with respect to $\bP$. Condition~\eqref{item-poles} of Proposition~\ref{prop:plancherel-j} says nothing unless $\chi$ is non-strictly positive with respect to $\bP'\in\cP(\bM)$. Now there is a minimal parabolic $\mathbf Q=\mathbf L\ltimes\mathbf U'$ containing all $\bP'\in\cP(\bM)$ such that $\chi$ is non-strictly positive with respect to $\bP'$. Now $\chi=\chi_1\chi_2$ where $\chi_1\in\Psi_u(M)$ and $\chi_2\in\Psi(L)$ is strictly positive with respect to $\mathbf Q$. Thus the intertwining operators may be written as
    \[
    J_{P'|P}=I_Q^G(J_{P'\cap L|P\cap L})\colon I_Q^G(I_{P\cap L}^L(\delta\otimes\chi_1)\otimes\chi_2)=I_P^G(\delta\otimes\chi)\to I_P^G(\delta\otimes\chi).
    \]
   By Lemma~\ref{lem:regular}, the intertwining operators $J_{P'\cap L|P\cap L}$ are regular at $\chi_1$, hence $J_{P'|P}$ is also regular. Thus we have an isomorphism
   \[
   \cJ^\land\simeq \big(\widehat{\mathcal O}_{\mathfrak m_\chi}\otimes \End_\C(I_P^G(V_{[\delta]})^K)\big)^{W_{[\delta]}},
   \]
   where now $W_{[\delta]}$ is a subgroup of $N_\bM(\mathbf L)/\mathbf L$. Let $\widetilde W_{[\delta]}$ be a finite covering of $W_{[\delta]}$ (closely related to the group $\widetilde R$ considered in \cite[\S2]{arthur-elliptic}.) Now modulo $\mathfrak m_\chi$ the $\widetilde W_{[\delta]}\times \cH(G,K)$-representation $I_P^G(\delta\otimes\chi)^K$ splits as: 
   \[I_P^G(\delta\otimes\chi)^K=\bigoplus_i \rho_i\boxtimes I_Q^G(\pi_i\otimes\chi_2)^K\]
   for tempered representations $\pi_i$ of $L$, direct summands of $I_{P\cap L}^L(\delta\otimes\chi_1)$.
   We can lift the decomposition to $\widehat\cO_{\mathfrak m_\chi}$:
   \[
   \widehat{\mathcal O}_{\mathfrak m_\chi}\otimes I_P^G(V_{[\delta]})^K=\bigoplus_i \widehat\rho_i\boxtimes I_Q^G(V_{[\pi_i]}\otimes\chi_2)^K,
   \]
   where now $\widehat\rho_i$ are $\widehat\cO_{\mathfrak m_\chi}$-semi-linear representations of $\widehat W_{[\delta]}$. Thus we have an isomorphism of $\widehat\cO_{\mathfrak m_{W_{[\delta]}\chi}}$-algebras
   \[
   \cJ^\land\simeq \bigoplus_{i,j}\Hom_{\widetilde W_{[\delta]}}(\rho_i,\rho_j)\boxtimes \Hom_\C(I_Q^G(V_{[\pi_i]}\otimes\chi_2),I_Q^G(V_{[\pi_j]}\otimes\chi_2)).
   \]
   When $i=j$ by Schur's lemma we have $\Hom_{\widetilde W_{[\delta]}}(\rho_i,\rho_j)=\widehat\cO_{\mathfrak m_{W_{[\delta]}\chi}}$, so the corresponding summand is $\widehat\cO_{\mathfrak m_{W_{[\delta]}\chi}}\boxtimes\End_\C(I_Q^G(V_{[\pi_i]}\otimes\chi_2))$. This gives the surjection
   \begin{equation}\label{eq:surjection}
   \cJ^\land/\mathfrak m_{W_{[\delta]}\chi}\to \bigoplus_i\End_\C(I_Q^G(V_{[\pi]}\otimes\chi_2)).
   \end{equation}
   Now observe that for $i\ne j$ and
   \[
   \phi_{ij}\otimes\eta_{ij}\in \Hom_{\widetilde W_{[\delta]}}(\rho_i,\rho_j)\boxtimes \Hom_\C(I_Q^G(V_{[\pi_i]}\otimes\chi_2),I_Q^G(V_{[\pi_j]}\otimes\chi_2)),
   \]
   we have
   \[
   \phi_{ij}\otimes\eta_{ij}=[\mathrm{Id}_{\widehat\rho_i\boxtimes I_Q^G(V_{[\pi_i]}\otimes\chi_2)^K},\phi_{ij}\otimes\eta_{ij}],
   \]
   which shows \eqref{eq:surjection} induces an isomorphism on cocenters. Moreover, we also see that components corresponding to $i\ne j$ are in the radical of $\cJ^\land/\mathfrak m$, which proves the classification of irreducible $\cJ/\mathfrak m$-modules.
\end{proof}
\begin{example}
    As pointed out to me by Vasily Krylov, the surjection $\cJ(G,K)^\fd/\mathfrak m_{W_{[\delta]}\chi}\to \bigoplus \End_\C(I_Q^G(\pi\otimes\chi')^K)$ need not be an isomorphism. For example, consider the Harish-Chandra block $[\mathbf T,\triv]$ of $\bG=\mathrm{SL}_2$ and let $K=I$ be the Iwahori subgroup. Then unramified characters of $T$ are parameterized by $\alpha\in\C^\times$, with the Weyl group acting as $\alpha\mapsto\alpha^{-1}$. The intertwining operator $I_B^G(\alpha)\to I_B^G(\alpha^{-1})$ is given by $\mathrm{diag}(\frac{q-\alpha}{q\alpha-1},1)$. Thus completing at $\alpha=-1$, \[\cJ(G,I)^\land\simeq M_2(\C[\![\alpha+1]\!])^{C_2}\]
    is isomorphic to, letting $z=\alpha+2+\alpha^{-1}$, the $\C[\![z]\!]$-algebra
    \[
    \begin{pmatrix}
        \C[\![z]\!]& \C[\![z]\!]\\
         z\C[\![z]\!]& \C[\![z]\!]
    \end{pmatrix}\subset M_2(\C[\![z]\!]).
    \]
    Thus $\cJ/\mathfrak m$ is a four-dimensional algebra with basis $e_{11},e_{12},ze_{21},e_{22}$ and radical spanned by $e_{12}$ and $ze_{21}$. This is the same issue with the proof in \cite[\S2.8]{braverman-kazhdan} as pointed out in \cite{geometric-realization}.
\end{example}

As an immediate corollary, we obtain:
\begin{thm}\label{thm:classification-of-irreducibles}
    The irreducible modules of $\mathcal J(G)$ are exactly the representations $I_P^G(\sigma\otimes\chi)$, where $\mathbf P=\mathbf M\ltimes \mathbf U\subset \mathbf G$ is a parabolic subgroup, $\chi\colon M\to \C^\times$ is a strictly positive unramified character, and $\sigma$ is an irreducible tempered representation of $M$.
\end{thm}
\begin{rmk}
    For the Iwahori block, this is \cite[Theorem~4.2]{cells4}.
\end{rmk}

Recall the Langlands classification of irreducible representations of $G$:
\begin{prop}[{\cite[\S XI.2]{borel-wallach}, \cite[Theorem~4.1]{silberger:langlands-quotient}}]\label{prop:langlands-classification}
    For any irreducible representations of $\pi$ there is a tuple $(\mathbf P,\sigma,\chi)$, called a standard triple, unique up to $G$-conjugacy, where $\mathbf P=\mathbf M\ltimes\mathbf U\subset\mathbf G$ is a parabolic subgroup, $\sigma\in\Rep_t(M)$ is irreducible, and $\chi\in\Psi(M)$ is strictly positive (with respect to $\mathbf P$), such that $\pi$ is the unique irreducible quotient of $I_P^G(\sigma\otimes\chi)$. Moreover, the classes $[I_P^G(\sigma\otimes\chi)]$ form a basis of the Grothendieck group $R(G)$ of admissible finite-length representations of $G$.
\end{prop}
Thus, Theorem~\ref{thm:classification-of-irreducibles} has the following corollary:
\begin{cor}\label{cor:grothendieck-group-iso}
    The pullback homomorphism $R(\cJ)\to R(\cH)$ induced from the inclusion $\cH\hookrightarrow\cJ$ is an isomorphism.
\end{cor}
\begin{proof}
    By Theorem~\ref{thm:classification-of-irreducibles} and Proposition~\ref{prop:langlands-classification} the classes $[I_P^G(\sigma\otimes\chi)]$ form a basis of both $R(\cJ)$ and $R(\cH)$, where $(\mathbf P,\sigma,\chi)$ are standard triples.
\end{proof}

\section{Trace Paley-Wiener theorem for the asymptotic Hecke algebra}\label{sec:trace-pw}
Let us recall the trace Paley-Wiener theorem for $\mathcal H(G)$: 
\begin{prop}[{\cite[Theorem~1.2]{trace-pw}}]\label{prop:trace-pw-H}
    Let $R(G)$ be the Grothendieck group of admissible finite-length $G$-modules and let $\Hom_{\reg}(R(G),\C)$ be the vector space of homomorphisms $\varphi\colon R(G)\to \C$ such that:
    \begin{itemize}
        \item there exists a compact open subgroup $K$ of $G$ such that $f([\pi])=0$ for all representations $\pi\in\Rep(G)$ with no $K$-fixed vectors; and
        \item for any parabolic subgroup $\mathbf P=\mathbf M\ltimes\mathbf U\subset \mathbf G$ and representation $\sigma\in\Rep(M)$, the function $f([I_P^G(\sigma\otimes\chi)])$ is a regular function of $\chi\in \Psi(M)$.
    \end{itemize}
    Then trace provides an isomorphism $\cH/[\cH,\cH]\simeq\Hom_{\reg}(R(G),\C)$.
\end{prop}
We prove an analogous statement for the asymptotic Hecke algebra $\mathcal J(G)$:
\begin{thm}\label{thm:trace-pw-J}
    Let $R(\mathcal J)$ be the Grothendieck group of finite-length $\mathcal J$-modules and let $\Hom_{\reg}(R(\mathcal J),\C)$ be the set of homomorphisms $\varphi\colon R(\mathcal J)\to \C$ such that:
    \begin{itemize}
        \item there exists a compact open subgroup $K$ of $G$ such that $f([\pi])=0$ for all representations $\pi\in\mathcal J\mathrm{-mod}$ with no $K$-fixed vectors; and
        \item for any parabolic subgroup $\mathbf P=\mathbf M\ltimes\mathbf U\subset \mathbf G$ and $\sigma\in\Rep^t(M)$, the function $f([I_P^G(\sigma\otimes\chi)])$ extends to a regular function of $\chi\in \Psi(M)$.
    \end{itemize}
\end{thm}
\begin{proof}[Proof of Theorem~\ref{thm:trace-pw-J}]
    First of all, we can show that for $f\in\cJ$ if $f([I_P^G(\sigma\otimes\chi)])=0$ for any standard triple $(\bP,\sigma,\chi)$ then $f\in[\cJ,\cJ]$.
    It suffices to prove the proposition when $f\in\cJ(G,K)^\fd$ where $\fd=[\bM,\delta]$ is a Harish-Chandra block and $K\subset G$ is a compact open subgroup. It suffices to check that $f\mod\mathfrak m_\chi$ is in $[\cJ(G,K)^\fd/\mathfrak m_\chi,\cJ(G,K)^\fd/\mathfrak m_\chi]$, since $\cJ(G,K)^\fd$ is a finitely generated $\cO(\Psi(M)\delta)$-module. This is shown in Proposition~\ref{prop:modulo-m}.

    Thus $\cJ/[\cJ,\cJ]\to\Hom(R(\cJ),\C)$ is injective. The image is in $\Hom_{\reg}(R(\cJ),\C)$. But by Proposition~\ref{cor:grothendieck-group-iso} we have $R(\cJ)\simeq R(G)$, which gives an isomorphism
    \[
    \Hom_{\reg}(R(\cJ),\C)\simeq \Hom_{\reg}(R(G),\C).
    \]
    Thus $\cJ/[\cJ,\cJ]\to\Hom_{\reg}(R(\cJ),\C)$ is also surjective.
\end{proof}
As a corollary, we also obtain the proof of Bezrukavnikov, Braverman, and Kazhdan's conjecture Theorem~\ref{main-thm}:
\begin{proof}[Proof of Theorem~\ref{main-thm}]
    By Proposition~\ref{prop:trace-pw-H} and Theorem~\ref{thm:trace-pw-J} there is a commutative diagram of isomorphisms
    \[\begin{tikzcd}
        \mathcal H/[\mathcal H,\mathcal H] \arrow{r} \arrow{d}{\simeq} & \mathcal J/[\mathcal J,\mathcal J] \arrow{d}{\simeq} \\
\Hom_{\reg}(R(\mathcal H),\C) \arrow{r}{\simeq}& \Hom_{\reg}(R(\mathcal J),\C).
    \end{tikzcd}\]
    Thus the top morphism must also be an isomorphism.
\end{proof}

\bibliographystyle{amsalpha}
\bibliography{bibfile}
\end{document}